\documentclass[a4paper,11pt]{article}
\usepackage[left=3cm,top=3cm,right=3cm,nohead,foot=1cm]{geometry}

\usepackage[plainpages=false]{hyperref}
\usepackage{amsfonts,latexsym,rawfonts,amsmath,amssymb,amsthm,mathrsfs}
\usepackage{amsmath,amssymb,amsthm,amsfonts,latexsym,lscape,rawfonts}
\usepackage{graphicx}
\usepackage{times}
\numberwithin{equation}{section}
\usepackage[all]{xy}
\usepackage{eufrak}
\usepackage{makeidx}         
\usepackage{graphicx,psfrag}
\usepackage{mathptmx}

\usepackage{array,tabularx}

\usepackage{setspace}
\input{xy}
\xyoption{all}

\newtheorem{thm}{Theorem}[section]

\newtheorem{lem}[thm]{Lemma}

\newtheorem{prop}[thm]{Proposition}

\theoremstyle{remark}
\newtheorem{rmk}[thm]{Remark}

\theoremstyle{definition}

\title{On four-dimensional anti-self-dual gradient Ricci solitons }
\author{Xiuxiong Chen and Yuanqi Wang}

\date{}

\begin{document}
\maketitle
\begin{abstract}
In this note  we prove that any  four-dimensional half conformally flat  gradient steady Ricci soliton must be  either Bryant's soliton or Ricci flat.
We also classify four-dimensional half conformally flat  gradient shrinking Ricci solitons with bounded curvature.
\end{abstract}

\section{Introduction.}
In 1982, R.Hamilton introduced the Ricci flow
in \cite{H3}, where he deforms a Riemannian metric in the direction
 of its Ricci tensor multiplied by $-2$ \[ \frac{\partial g_{ij}}{\partial t}=-2Ric_{ij}.\]
 The Ricci flow equation gives a canonical way of deforming an arbitrary
 metric to a critical metric (an Einstein metric in particular). It has been remarkably successful program over years, in particular,
 the seminal work of G. Perelman on the Poincare conjecture.
 More recently, S. Brendle and R. Schoen prove the Differentiable Sphere Theorem for pointwise $1/4$-pinched manifolds (cf. \cite{Brendle-book}).
 Note the classical sphere theorem need to assume a global $1/4$ pinched condition on curvature.
 \\
In this paper, we will study the  Ricci solitons. Following R.  Hamilton \cite{Ha95F},
 a smooth complete Riemannian manifold  $(M,g)$
is called a gradient  Ricci soliton if there exists
a smooth function $f$ on $M$  and a constant $\rho$ such that the following identity holds:
$$R_{ij}+\nabla_i\nabla_jf=\rho g. $$
 We say it's  shrinking if $\rho>0$, steady if $\rho=0$ or expanding if $\rho<0$.\\
Gradient Ricci solitons are  extensions of Einstein metrics and it plays an important role in the singularity analysis of
Hamilton's Ricci flow. \\
 In the 2-dim case, Hamilton \cite{Ha88} discovered the {\it cigar soliton}, which is
 a complete noncompact gradient steady
soliton on $\mathbb R^2$. The metric is
$$ g_{cigar}=\frac{dx\otimes dx +dy\otimes dy}{1+x^2+y^2}.$$
This soliton is positively curved and looks like a cylinder of finite
circumference at infinity.  Furthermore, Hamilton \cite{Ha88} showed that  the only
complete steady soliton on a 2-dimensional manifold with
bounded scalar curvature $R$ which attains its maximum
at a point is  the cigar soliton up to a scaling. For $n\geq 3$, R.Bryant  showed in \cite{Bryant} that
there exists  an unique complete rotationally symmetric gradient Ricci
soliton on $\Bbb R^n$ up to scaling.
 The Bryant soliton has positive curvature operator, and curvature tensor  decays in the order of $\frac{1}{r}$. \\
In 3-dimensional case, Perelman claimed in \cite{Pere02} without proof that the Bryant soliton is the
only 3-dim complete noncompact positively curved $\kappa$-noncollapsed gradient steady
soliton. It's important to understand the picture hidden behind. On this direction there are  a lot of interesting works, for example the works of S.Brendle \cite{Brendle}, B.Chen \cite{BChen} and Cao-Cheng \cite{Cao Cheng}. For $n\geq 4$,  assuming the Weyl-tensor vanishes, Cao and Chen proved in \cite{Cao Cheng} that for $n\geq 4$,  any $n$-dimensional   noncompact
gradient  steady Ricci soliton with vanishing Weyl-tensor and positive curvature operator is isometric to
the Bryant soliton (up to a scaling).\\
The classification of the complete gradient shrinking solitons is also
important since  they are the blown up limits of type-I singularities of compact Ricci flows (see \cite{Na}) and blown down limits of ancient solutions of certain types(see Perelman \cite{Pere02}). On this direction there are  a lot of interesting works, for example the works of T.Ivey \cite{Iv}, Perelman \cite{Pere}, Naber\cite{Na}, Ni-Wallach\cite{NW}, Gu-Zhu\cite{GZ}, Kotschwar\cite{Kotschwar},
Peterson-Wylie\cite{PW}.

The problem  we consider is the classification of 4-dimensional half conformally flat steady and shrinking gradient Ricci solitons. Half conformally flat metrics are also known as self-dual or anti-self-dual if $W^- = 0$
or $W^+ = 0$ , respectively. Without loss of generality, we can assume the soliton is anti-self-dual. Our main results are :
\begin{thm}\label{classification of steady soliton}Any 4-dim complete  gradient steady  Ricci soliton with $W^+ = 0$
 must be isometric to one of the following two types:
\begin{enumerate}
  \item The Bryant Soliton (up to a scaling).
  \item A  manifold which is anti-self-dual and  Ricci flat.
\end{enumerate}
\end{thm}
\begin{thm}\label{classification of shrinking soliton}Any 4-dim complete   gradient shrinking  Ricci soliton with bounded curvature and $W^+ = 0$
 must be isometric to  finite quotients of $R^4,\ S^3\times R,\ S^4$ or $CP^2$.
\end{thm}
\begin{rmk}Obviously if we change the condition $W^+=0$ to $W^-=0$, the conclusions are exactly the same. The proofs of  Theorem \ref{classification of steady soliton} and \ref{classification of shrinking soliton} both originate from a wonderful idea in Cao-Chen \cite{Cao Cheng}. On Theorem \ref{classification of steady soliton}, note that we don't require the curvature to be bounded globally. In case (2) a complete picture is still lacking, one can see Remark \ref{Ricci flat antiselfdual manifolds}.\\
 On Theorem \ref{classification of shrinking soliton}, the reason we need the curvature to be  uniformly bounded is that we want the metric to be analytic. The fact that Ricci flows with uniformly bounded curvature are analytic when $t>0$ is proved by Bando in \cite{Bando}. We strongly conjecture that Ricci Solitons should all be analytic without any curvature condition.\\
 Notice that Theorem \ref{classification of shrinking soliton} implies that the K\"ahler shrinking Ricci soliton on $CP^2\sharp\overline{CP^2}$(  found by Koiso \cite{Koiso} and  Cao \cite{Cao's contruction of Soliton}), and  the soliton on $CP^2\sharp 2\overline{CP^2}$( found by Wang and Zhu   \cite{Wang-Zhu})  are not either anti-self-dual or self-dual.
\end{rmk}
\textbf{Acknowledgements}: Both authors  are grateful to
Bing Wang for observing Lemma \ref{unique continuation of steady soliton} and lots of helpful discussions. The second author  wants to thank Prof Gang Tian for helpful suggestions. The second author also benefits a lot from communication with Prof Xiaohua Zhu.

\section{Proof of Theorem \ref{classification of steady soliton}.}
From now on we  assume $(M,g)$ is anti-self-dual, i.e  for some corrrect oriented orthonomal basis  $e_1,\ e_2,\ e_3,\ e_4$ we have
 \begin{equation}\label{definition of anti-self-dual}W^{-}_{12kl}=-W^{-}_{34kl},\ W^{-}_{13kl}=-W^{-}_{42kl},\ W^{-}_{14kl}=-W^{-}_{23kl}\end{equation}
  In this section we assume that $(M,g,f)$ satisfies $R_{ij}+\nabla_i\nabla_j f =0$. Our crucial observation is that for Ricci solitons,  the condition $W^{+}=0$ actually implies that the full Weyl tensor vanishes at those points where $\nabla f\neq 0$.
\begin{lem}\label{Ric has at most two eigenvalues}  Suppose $(M,g)$ is an anti-self-dual steady  Ricci soliton. Suppose $q\in M$ is not a critical point of $f$ i.e  $|\nabla f|(q)\neq 0$. Then at the point $q$,    $e_1=\nabla f /|\nabla f |$ is an eigenvector of $Ric$.
Furthermore, there exists two constants $\lambda$ and $\mu$ such that for any orthonormal basis $e_2,\ e_3,\ e_4$ tangent to the level surface of $f$  at $p$, the
Ricci tensor has the following properties:
\begin{enumerate}
  \item  $Ric(e_1, e_1)=\lambda$,

  \item $Ric(e_i, e_j)=\mu\delta_{ij}, \  i, j=2, 3, 4$.
  \item
 $Ric(e_1, e_j)=R_{1j}=0, \ j=2, 3, 4$,
\end{enumerate}
\end{lem}
\begin{proof}{of lemma \ref{Ric has at most two eigenvalues}:}\\ Let $n=4$. First we recall the definition of the Cutton tensor:
 $$C_{ijk}=-\frac{n-3}{n-2}W_{ijkl,l}$$ and
$$C_{ijk}=R_{jk,i}-R_{ik,j}-\frac{1}{2(n-1)}(R_ig_{jk}-R_jg_{ik})$$
Using  $R_{jk}+f_{jk}=0$, we have
\begin{eqnarray*}& &C_{ijk}=f_{ikj}-f_{jki}-\frac{1}{2(n-1)}(R_ig_{jk}-R_jg_{ik})
\\& &=-R_{ijkl}f_l-\frac{1}{2(n-1)}(R_ig_{jk}-R_jg_{ik})
\end{eqnarray*}
Using \begin{eqnarray*}& &R_{ijkl}
\\&=&W_{ijkl}+\frac{1}{n-2}(R_{ik}g_{jl}+R_{jl}g_{ik}-R_{il}g_{jk}-R_{jk}g_{il})
\\& &-\frac{R}{(n-1)(n-2)}(g_{jl}g_{ik}-g_{il}g_{jk})\end{eqnarray*}
We get
\begin{eqnarray*}& & C_{ijk}=
-W_{ijkl}f_l-\frac{1}{n-2}(R_{ik}g_{jl}+R_{jl}g_{ik}-R_{il}g_{jk}-R_{jk}g_{il})f_l\\&+&\frac{R}{(n-1)(n-2)}(g_{jl}g_{ik}-g_{il}g_{jk})f_l
-\frac{1}{2(n-1)}(R_ig_{jk}-R_jg_{ik})
\\&=&-W_{ijkl}f_l-\frac{1}{n-2}(R_{ik}f_j-R_{jk}f_i)+\frac{1}{2(n-2)}(g_{jk}R_i-g_{ik}R_j)\\&+&\frac{R}{(n-1)(n-2)}(f_jg_{ik}-f_ig_{jk})
-\frac{1}{2(n-1)}(R_ig_{jk}-R_jg_{ik})
\end{eqnarray*}
Using $2R_{jl}f_l=R_j$, we get
\begin{eqnarray*}& &C_{ijk}
\\&=&-W_{ijkl}f_l-\frac{1}{n-2}(R_{ik}f_j-R_{jk}f_i)+\frac{1}{(n-2)(n-1)}(g_{jk}R_{il}f_l-g_{ik}R_{jl}f_l)\\&+&\frac{R}{(n-1)(n-2)}(f_jg_{ik}-f_ig_{jk})
\end{eqnarray*}
Since   $\nabla f|_{p}\neq0$, suppose $f(p)=c$.
We denote \begin{eqnarray*}& &B_{ijk}
\\&=&-\frac{1}{n-2}(R_{ik}f_j-R_{jk}f_i)+\frac{1}{(n-2)(n-1)}(g_{jk}R_{il}f_l-g_{ik}R_{jl}f_l)\\&+&\frac{R}{(n-1)(n-2)}(f_jg_{ik}-f_ig_{jk})\end{eqnarray*}
Choose  $e_1,e_2,e_3,e_4\ $ which diagonalize $Ric$ at $q$. Assuming $W^{+}=0$, we get that:
\begin{eqnarray*}& &B_{12k}+B_{34k}
\\&=&C_{12k}+C_{34k}+(W_{12kl}+W_{34kl})f_l
\\&=&-\frac{n-3}{n-2}(W_{12kl}+W_{34kl})_{,l}+(W_{12kl}+W_{34kl})f_l
\\&=&0\end{eqnarray*}
Similarly we also have $B_{13k}+B_{42k}=0,\ B_{14k}+B_{23k}=0$.
By definition we see that $B_{ijk}=0$ if $i,j,k$ are mutually different. So we have:
$$B_{iji}\equiv 0, \textrm{for all}\ i,j=1,2,3,4.$$
Furthermore by definition of $B$ again we have  $B_{iij}=0$. Then we see that $B\equiv 0$.\\ Now we denote the eigenvalue of $Ric$ with respect to $e_k,\ k= 1,2,3,4$ to be $\lambda_k,\ k=1,2,3,4.$ respectively. \\ Now let $n=4$,  take $B_{121}$ as an example,  we have
$$0=-\frac{1}{2}(\lambda_1f_2)-\frac{1}{6}f_2\lambda_2+\frac{R}{6}f_2=\frac{f_2}{6}(R-3\lambda_1-\lambda_2)$$
Which means we should have $$\frac{f_i}{6}(R-3\lambda_j-\lambda_i)=0,\ \textrm{for all}\ i\neq j.$$
Thus if $\nabla f$ has more than two nonzero component, we have $\lambda_1=\lambda_2=\lambda_3=\lambda_4$.
If $\nabla f$ has only one nonzero component, say $f_1$, then $\lambda_2=\lambda_3=\lambda_4$. In either case, $\nabla f$ is an eigenvector of $Ric$ and $Ric|_{(\nabla f)^{\perp}}=\mu \delta{ab}$.
\end{proof}
Next let $K_c=\{f(p)=c\}$ be the level set of $f$. The proof of the next lemma is essentially the same as lemma 3.3 in \cite{Cao Cheng}.
\begin{lem}\label{constancy of mean curvature} Suppose $(M,g)$ is an anti-self-dual steady gradient Ricci soliton. Suppose $q\in M$ is not a critical point of $f$ i.e  $|\nabla f|(q)\neq 0$. Then near the point $p$ we have that the second fundamental form $\Pi_{ij}$ of $K_c$ is a constant multiple of the of the metric i.e  $\Pi_{ij}=\frac{\eta}{3} g_{ij}$, in which $\eta$ is the mean curvature. Furthermore,  $\eta$ is constant on $K_c$.
\end{lem}
\begin{proof}{of Lemma \ref{constancy of mean curvature}:}\\ From Lemma \ref{Ric has at most two eigenvalues} we can set $e_1=\frac{\nabla f}{|\nabla f|}$. First notice that $\nabla f\neq 0$ at $q$ implies that there exists a coordinate system $\psi_q$ near $q$ such that $f$ is one of coordinate functions of $\psi_q$.  Suppose $i,j\neq 1$, restricted to the distribution $\nabla f^{\bot}$, we compute \begin{eqnarray*}& &\Pi_{ij}=\frac{\nabla\nabla f}{|\nabla f|}=\frac{-R_{ij}}{|\nabla f|}
=\frac{-\frac{R-R_{11}}{3}}{|\nabla f|}g_{ij}\end{eqnarray*}
\end{proof}
Thus the mean curvature satisfies $\eta=\frac{-R+R_{11}}{|\nabla f|}$.

Next we show that $\eta$ is a constant along fiber direction. Since \ref{Ric has at most two eigenvalues} holds, then direct computation shows that  for  $i, j, k=2, 3, 4$ we have
$$R_{1ijk}=-\nabla_j^{T}\Pi_{ki}+\nabla_k^{T}\Pi_{ji}.$$ Contracting $i$ and $k$, we obtain
$$0=R_{1j}=-\nabla_j^{T} \eta+\nabla_k^{T}\Pi_{jk}=-\frac{2}{3}\nabla_j^{T} \eta,$$
Since j is arbitrary,  $\eta$ is constant over $K_c$.

\begin{lem}\label{nondegenerace implies W vanishes}Let $(M,g,f)$ be a 4-dim  steady Ricci soliton   and $W^+=0$. Suppose $q\in M$ and $|\nabla f|^2(q)\neq 0$, then  $W= 0$ at $q$.
\end{lem}
\begin{proof}{of Lemma \ref{nondegenerace implies W vanishes}:}\\
From the proof of lemma \ref{constancy of mean curvature} we have known what $R_{1ijk}\equiv 0$ at $q$, if $\{i, j, k\}=\{2, 3, 4\}$. Recall we have the curvature decomposition: \begin{equation}\label{curvature decomposition} R_{ijkl}=\frac{1}{2}[A_{ik}g_{jl}+g_{ik}A_{jl}-A_{il}g_{jk}-g_{il}A_{jk}]+W_{ijkl}.\end{equation}
 The $A$ is the Schouten tensor:
 \begin{equation*}\label{Definition of schouten tensor}A=\frac{2}{n-2}(R_{ij}-\frac{R}{2n-2}g_{ij})
    \end{equation*}

 Since $e_1,\ e_2,\ e_3,\ e_4$ are all eigenvectors of $Ric$, then they are also eigenvectors of $A$. Then we have
  $$(A\otimes g)_{1abc}=0, \textrm{for all}\ a,b,c=\{2,3,4\}.$$
 Thus  from \ref{curvature decomposition} and our assumption that $W^+=0$, we have:
\begin{equation}\label{W^- vanishes}W^{-}_{1abc}=0, \textrm{for all}\ a,b,c=\{2,3,4\}.\end{equation}
A well known property of $W^{-}$ says that for any $k,l=1,2,3,4$, we have  \begin{equation}\label{property of W^-}W^{-}_{12kl}=-W^{-}_{34kl},\ W^{-}_{13kl}=-W^{-}_{42kl},\ W^{-}_{14kl}=-W^{-}_{23kl}\end{equation}
 Thus combining \ref{W^- vanishes} and \ref{property of W^-} we have
$$W^{-}\equiv 0,$$ which means $g$ actually has vanishing Weyl tensor at $q$.
\end{proof}
Next we should quote two statements which are very important to us.
\begin{thm}{(Cao-Chen\cite{Cao Cheng})}\label{steady solitons with vanishing Weyl tensor} Suppose $n\geq 4$. Any complete n-dim gradient steady  soliton with vanishing Weyl tensor must be either flat  or isometric to the Bryant soliton.
\end{thm}

\begin{prop}{(B.L.Chen\cite{BChen})} \label{ancient solutions have nonnegative R} Any  complete ancient
solution to the Ricci flow has nonnegative scalar curvature everywhere.
\end{prop}
Based on Proposition \ref{ancient solutions have nonnegative R}, Bing Wang has an observation.
\begin{lem}{(Bing Wang \cite{Bing Wang})}\label{unique continuation of steady soliton} Suppose $(M,g,f)$ is a gradient steady Ricci soliton i.e:
$$Ric+\nabla\nabla f=0.$$
If  \ $|\nabla f|^2=0$ over an open set, then $f\equiv constant$ over the whole manifold $M$ and $g$ is Ricci flat.
\end{lem}

\begin{proof}{of Lemma \ref{unique continuation of steady soliton}:}\\ Suppose
 $|\nabla f|^2$ vanishes in a ball $B_{v}$. By definition of steady soliton,  we have $R
\equiv 0$ in the $B_{v}$. Since $R+|\nabla f|^2=Constant$ over the manifold $M$, then over $B_{v}$ we obviously have $C=0$. Thus
$$R+|\nabla f|^2=0 \ \textrm{over the manifold M}.$$
But according to B.Chen's result \ref{ancient solutions have nonnegative R}, any steady Ricci soliton (not necessarily with bounded curvature) has $R\geq 0$, then $R\equiv 0\equiv |\nabla f|^2$.
Which means $f\equiv C_1$ and $(M,\ g)$ is Einstein.
\end{proof}
Now we are ready to prove the first main theorem.
\begin{proof}{\ of Theorem \ref{classification of steady soliton}:}\\ First we assume the set $\{|\nabla f|^2\neq 0\}$ is dense in $M$, then according to lemma \ref{nondegenerace implies W vanishes}, $W=0$ over a dense set. By continuity of the Weyl tensor we have $W=0$ everywhere. By Cao-Chen's  classification \ref{steady solitons with vanishing Weyl tensor}, we are in  case (1) or (2) in Theorem \ref{classification of steady soliton}.\\ \\
If the above does not hold, then
 $|\nabla f|^2$ vanishes in a ball $B_{v}$. By lemma \ref{unique continuation of steady soliton} $f$ is a constant over $M$ and $g$ is Ricci flat. Therefore we are in case (2) of Theorem \ref{classification of steady soliton}.
\end{proof}
\begin{rmk} One should notice that if the curvature is uniformly bounded, the fact $Ric\equiv 0$ over the whole manifold $M$ provided that $|\nabla f|^2$ vanishes in a ball is also directly implied by lemma \ref{analyticity of Solitons}.
\end{rmk}
\begin{rmk}\label{Ricci flat antiselfdual manifolds} One might be interested in the 2nd case in Theorem \ref{classification of steady soliton}. Actually in the special case that $(M,g)$ is $ALE$ and hyper-K\"ahler, Kronheimer obtained a very nice classification (see  \cite{P.B.Kronheimer}). The Eguchi-Hanson metric is  in Kronheimer's classification. Note that hyper-K\"ahlerity implies  $(M,g)$ is half conformal flat and Ricci flat. But if $(M,g)$ is not $ALE$, there are still a lot of possibilities. In Tian-Yau \cite{Tian Yau1}, it's showed explicitly that there exists a complete noncompact 4-dim K\"ahler manifold which is Ricci flat  and has volume growing in the rate of $r^{\frac{4}{3}}$. Note again K\"ahler manifolds with zero Ricci curvature is automatically half conformally flat.  Furthermore, Tian-Yau (see \cite{Tian Yau1} and \cite{Tian Yau2}) constructed a vast class of complete noncompact K\"ahler manifolds of finite topological type. On the other hand,  M. Anderson,
P. Kronheimer and Le Brun  constructed such examples with
infinite topological type. There should be other constructions which the authors don't know.
\end{rmk}
\section{Proof of Theorem  \ref{classification of shrinking soliton}.}
  Again we  assume $(M,g)$ is anti-self-dual.  The basic equation is
   $$Ric+\nabla\nabla f=g.$$
   The situation is essentially the same as the proof of Theorem \ref{classification of steady soliton}, the difference is that we have to apply lemma \ref{analyticity of Solitons} thus we need the curvature to be uniformly bounded.
   First by the same proof  of lemma \ref{Ric has at most two eigenvalues} we also have:
   \begin{lem}\label{Ric has at most two eigenvalues shrinking } Suppose $(M,g)$ is an anti-self-dual gradient shrinking Ricci soliton. Suppose $q\in M$ is not a critical point of $f$ i.e  $|\nabla f|(q)\neq 0$. Then at the point $q$,    $e_1=\nabla f /|\nabla f |$ is an eigenvector of $Ric$.
Furthermore, there exists two constants $\lambda$ and $\mu$ such that for any orthonormal basis $e_2,\ e_3,\ e_4$ tangent to the level surface of $f$  at $p$, the
Ricci tensor has the following properties:
\begin{enumerate}
  \item  $Ric(e_1, e_1)=\lambda$,

  \item $Ric(e_i, e_j)=\mu\delta_{ij}, \  i, j=2, 3, 4$.
  \item
 $Ric(e_1, e_j)=R_{1j}=0, \ j=2, 3, 4$,
\end{enumerate}

\end{lem}
Therefore we also have the  following lemma by the same proof of lemma \ref{nondegenerace implies W vanishes}:
   \begin{lem}\label{nondegenerace implies W vanishes of shrinking case }Let $(M,g,f)$ be a 4-dim complete  gradient shrinking Ricci soliton  with bounded curvature and $W^+=0$. Suppose $q\in M$ and $|\nabla f|^2(q)\neq 0$, then  $W= 0$ at $q$.
\end{lem}
Next we recall a result of Bando in \ref{analyticity of Ricci flows} on analyticity of Ricci flows.
\begin{thm}{(Bando)}\label{analyticity of Ricci flows}Suppose $\{M,g(t),\ t\in[0,T]\}$ is a complete Ricci flow with bounded curvature, then $g(t)$ is analytic in polar coordinates for $t>0$.
\end{thm}
Bando's result directly implies the following lemma.
\begin{lem}\label{analyticity of Solitons}Let $(M,g,f)$ is a  complete Ricci soliton with bounded curvature, i.e: $$R_{ij}+\nabla_i\nabla_j f =\rho g_{ij}\ \ \textrm{for some constant $\rho$ and smooth function $f$}.$$
Suppose there exists an open set $\Omega$  in which  $|\nabla f|=0$, then $f$ is a constant  over the whole manifold $M$ and $(M,g)$ is Einstein.
\end{lem}
\begin{proof}{of Lemma \ref{analyticity of Solitons}:}\\ Obviously we have  $f=C_1$ in  $\Omega$, $C_1$ is a constant.  If the lemma does not hold, there exist a short geodesic $\gamma(s),\ s\in[0,s_0]$ with arclength parameter and a number $\upsilon$ such that
 $$f[\gamma(s)]=C_1, \textrm{when}\ {s<\upsilon}; \ f[\gamma(s_0)]\neq C_1 .$$
  Notice for all $p\in M$ and    any geodesic $\gamma$ passing through $p$, $\gamma$ is a straight line in the polar coordinate centered at $p$ and therefore analytic.  By the bounded curvature assumption and Bando's result, we know that $Ric(\dot{\gamma(s)},\ \dot{\gamma(s)} )$ is analytic in $s$, which means that   $f^{''}(s)$ is analytic in $s$.
 Then  $f$ is also analytic in $s$ which implies $f(s)\equiv C_1$ for all $s\in[0,s_0]$. This is a contradiction.
\end{proof}
Before proving Theorem \ref{classification of shrinking soliton}, we shall quote the main theorem in \cite{Zh2}.
\begin{thm}{(Z-H Zhang)}\label{Shrinking soliton with vanishing Weyl tensor}When $n\geq 4$,
any complete gradient shrinking soliton with vanishing Weyl tensor
must be the finite quotients of $R^n$, $S^{n-1}\times R$, or $S^n$.
\end{thm}

Now we are ready to prove Theorem \ref{classification of shrinking soliton}.
\begin{proof}{\ of Theorem \ref{classification of shrinking soliton}:}\\Similarly, we first assume the set $\{|\nabla f|^2\neq 0\}$ is dense in $M$ , then according to lemma \ref{nondegenerace implies W vanishes}, $W=0$ over a dense set. By continuity of the Weyl tensor we have $W=0$ everywhere. By Zhang's classification \ref{steady solitons with vanishing Weyl tensor}, $(M,g)$ must be the finite quotients of $R^4$, $S^3\times R$ or $S^4$.\\ \\
If the above does not hold, then
 $|\nabla f|^2$ vanishes in a ball $B_{v}$. By lemma \ref{analyticity of Solitons},  we have  that  $Ric\equiv g$ over the whole manifold $M$. Therefore $M$ is compact Einstein with positive scalar curvature and $W^+=0$, then it's necessarily isometric to $S^4$ or $CP^2$ according to a theorem of Hitchin(see theorem 13.30 of \cite{Besse}).
 Thus Theorem \ref{classification of shrinking soliton} is true.
\end{proof}

Xiuxiong Chen, Department of Mathematics, University of Wisconsin-Madison, Madison,
WI 53706, USA;\ \ xiu@math.wisc.edu\\ \\
Yuanqi Wang, Department of Mathematics, University of Wisconsin-Madison, Madison,
WI, 53706, USA;\ \ wangyuanqi1982@gmail.com

\end{document}